\documentclass{amsart}   	
\usepackage{geometry}                		
\usepackage{graphicx}				
\usepackage{amssymb,bbm}

\usepackage{hyperref}

\allowdisplaybreaks
\DeclareSymbolFont{bbold}{U}{bbold}{m}{n}
\DeclareSymbolFontAlphabet{\mathbbold}{bbold}
\newcommand{\ind}{\mathbbold{1}}

\newcommand{\C}{\mathbb{C}}

\newcommand{\N}{\mathbb{N}}

\newcommand{\R}{\mathbb{R}}
\newcommand{\Z}{\mathbb{Z}}

\newcommand{\bP}{\mathbb{P}}
\newcommand{\bE}{\mathbb{E}}

\newcommand{\cD}{\mathcal{D}}
\newcommand{\cE}{\mathcal{E}}
\newcommand{\cF}{\mathcal{F}}
\newcommand{\cK}{\mathcal{K}}
\newcommand{\cL}{\mathcal{L}}
\newcommand{\cS}{\mathcal{S}}

\DeclareMathOperator\ess{ess}

\newcommand{\abs}[1]{\left| #1 \right|}
\newcommand{\norm}[1]{\left \lVert #1 \right \rVert}

\newcommand{\brac}[1]{\left\{ #1 \right\}}

\theoremstyle{plain} \newtheorem{theorem}{Theorem}[section]
\theoremstyle{plain} \newtheorem{proposition}[theorem]{Proposition}
\theoremstyle{plain} \newtheorem{lemma}[theorem]{Lemma}
\theoremstyle{plain} \newtheorem{corollary}[theorem]{Corollary}
\theoremstyle{definition} \newtheorem{definition}[theorem]{Definition}
\theoremstyle{definition} \newtheorem{notation}[theorem]{Notation}
\theoremstyle{remark} \newtheorem{remark}[theorem]{Remark}
\theoremstyle{remark}

\title{Rotatable Random Sequences in Local Fields}

\author[S.N. Evans]{Steven N. Evans}
\address{Department of Statistics \#3860\\
 367 Evans Hall \\
 University of California \\
  Berkeley, CA  94720-3860 \\
   USA} 
\email{evans@stat.berkeley.edu}

\author[D. Raban]{Daniel Raban}
\address{UCLA Mathematics Department \\
Box 951555 \\
Los Angeles, CA 90095-1555 \\
USA}
\email{danielraban@math.ucla.edu}

\thanks{SNE supported in part by NSF grant DMS-1512933
and NIH grant 1R01GM109454-01.}

\subjclass[2010]{primary 60B99; 60G09; secondary 12J25} 

\keywords{$p$-adic; $p$-series; exchangeable; spherically symmetric; Gaussian; total variation; completely monotone; nonnegative definite}
	
\date{\today}							

\numberwithin{equation}{section}

\begin{document}

\begin{abstract}
An infinite sequence of real random variables $(\xi_1, \xi_2, \dots)$ is said to be rotatable if every finite subsequence $(\xi_1, \dots, \xi_n)$ has a spherically symmetric distribution. A celebrated theorem of Freedman states that $(\xi_1, \xi_2, \dots)$ is rotatable if and only if $\xi_j = \tau \eta_j$ for all $j$, where $(\eta_1, \eta_2, \dots)$ is a sequence of independent standard Gaussian random variables and $\tau$ is an independent nonnegative random variable. Freedman's theorem is equivalent to a classical result of Schoenberg which says that a continuous function $\phi : \R_+ \to \C$ with $\phi(0) = 1$ is completely monotone if and only if 
$\phi_n: \R^n \to \R$ given by $\phi_n(x_1, \ldots, x_n) = \phi(x_1^2 + \cdots + x_n^2)$ is nonnegative definite for all $n \in \N$.  We establish the analogue of Freedman's theorem for sequences of random variables taking values in local fields using probabilistic methods  and then use it to establish a local field analogue of Schoenberg's result.  Along the way, we obtain a local field counterpart of an observation variously attributed to Maxwell, Poincar\'e, and Borel which says that if $(\zeta_1, \ldots, \zeta_n)$ is uniformly distributed on the sphere of radius $\sqrt{n}$ in $\R^n$, then, for fixed $k \in \N$, the distribution of $(\zeta_1, \ldots, \zeta_k)$ converges to that of a vector of $k$ independent standard Gaussian random variables as $n \to \infty$.
\end{abstract}

\maketitle

\section{Introduction}

An $n$-vector $\xi$ of real-valued random variables is \textit{rotatable} if $U \xi$ has the same distribution as $\xi$ for every $n \times n$ orthogonal matrix $U$; that is, the distribution of $\xi$ is spherically symmetric.
Similarly, an infinite sequence $(\xi_i)_{i \in \N}$ of real-valued random variables is rotatable if the vectors $(\xi_i)_{i \in [n]}$ are rotatable for every $n \in \N$ (here we use the notation $[n] := \{1,\ldots,n\}$). 
Because permutation matrices are orthogonal, it follows that a rotatable infinite sequence is exchangeable and hence, by de Finetti's theorem, distributed as a mixture of independent, identically distributed sequences.  
A famous result of Maxwell \cite{Maxwell_75, Maxwell_78} says that if a real random vector is spherically symmetric and has independent (necessarily identically distributed) entries, then the distribution of the entries is centered Gaussian.  
Combining these two observations makes plausible the celebrated theorem of Freedman \cite{MR0156369} (see, also, \cite{MR0287628, MR0343420, MR606622, MR606624, MR626767} that a rotatable infinite sequence of real-valued random variables is a scale mixture of sequences of independent standard Gaussian random variables; that is,
if $\xi$ is an infinite rotatable sequence, then $\xi = \tau \eta$, where  $(\eta_i)_{i \in \N}$ is a sequence of independent standard Gaussian random variables and $\tau$ is a nonnegative random variable that is independent of $\eta$.

We refer the reader to the {\em Historical and Bibliographical Notes} in \cite{MR2161313} for an indication of the later literature around Freedman's theorem and for remarks on the connection between this result and the classical theorem of Schoenberg \cite{MR1501980} which says that a continuous function $\phi : \R_+ \to \C$ with $\phi(0) = 1$ is completely monotone if and only if 
$\phi_n: \R^n \to \R$ given by $\phi_n(x_1, \ldots, x_n) = \phi(x_1^2 + \cdots + x_n^2)$ is nonnegative definite for all $n \in \N$.

Our primary goal, realized in Theorem~\ref{T:Freedman}, is to obtain an analogue of Freedman's theorem for sequences of random variables taking values in fields other than the fields of real and complex numbers.  
More specifically, we consider the \textit{local fields}; a local field is any locally compact, non-discrete topological  field other than $\R$ or $\C$ (a local field is, for some prime $p$, a finite algebraic extension of either the field of $p$-adic numbers or the field of Laurent series over the finite field of integers modulo $p$).   
We recall some facts about the structure of local fields and vector spaces over them in Section~\ref{S:local_fields}.
Just as in the real case, Freedman's theorem is equivalent to a structure theorem for nonnegative definite functions, and we present such a result in Section~\ref{S:Schoenberg}.

In order to search for a counterpart Freedman's theorem we need to have a parallel for the group of orthogonal matrices so that we can say what it means for a distribution to be ``spherically symmetric'' in the local field setting.  
The $n \times n$ orthogonal matrices are, of course, the matrices that preserve the Euclidean metric on $n$-dimensional space and so, denoting our local field by $\cK$,  we take as our parallel of the orthogonal matrices those matrices that are isometries of $\cK^n$, where $\cK^n$ is equipped with the natural metric described in Subsection~\ref{S:norms} -- see Subsection~\ref{S:orthogonality}.

Not surprisingly, our counterpart of Freedman's theorem also involves a suitable parallel of the class of Gaussian measures in a local field setting.  Such an analogue was considered in \cite{MR990478, MR1832433}.  
The idea is that, given one has a notion of spherical symmetry, one can take the property embodied in Maxwell's theorem to be the definition of Gaussianity for local fields.
We recall some of the elementary properties of Gaussian probability measures on local fields in Subsection~\ref{S:Gaussian} after we have laid some groundwork on Haar measure and Fourier theory on local fields in Subsection~\ref{S:Haar_measure} and Subsection~\ref{S:Fourier}, respectively.

\section{Local fields}
\label{S:local_fields}

From now on, let $\cK$ be a fixed local field.  
We refer the reader to  \cite{MR2444734, MR512894,  MR0487295} for in-depth treatments of various aspects of analysis on local fields.  
Any result that we state without a proof or a bibliographic citation may be found in these references.

\subsection{Basics}
\label{S:basics}

There is a distinguished real-valued mapping on $\cK$ which we denote by $\vert \cdot \vert$;
this map has the properties
\begin{align}
&\vert x \vert = 0 \iff x=0, \\
&\vert xy \vert = \vert x \vert \vert y \vert, \\
&\vert x+y \vert \le \vert x \vert \vee \vert y \vert,
\end{align}
and takes the values $\{0\} \cup \{ q^m : m \in \Z \}$, where $q = p^c$ for some prime $p$ and positive integer $c$.

A map with properties (1.1)-(1.3) is called a \textit{non-Archimedean valuation}. Property (1.3) is known as the {\it ultrametric inequality} or the {\it strong triangle inequality}. The mapping $(x,y) \mapsto \vert x-y \vert$ on $\cK \times \cK$ is a metric on $\cK$ which induces the topology of $\cK$. 
The metric space $\cK$ is a complete, totally disconnected, \textit{ultrametric space} under this metric.

Write $\cD$ for the closed unit ball $\{ x: \vert x \vert \le 1 \}$.
Choosing $\rho \in \cK$ so that $\vert \rho \vert =q^{-1}$,  we have
\[
\rho^k \cD = \{ x: \vert x \vert
\le q^{-k} \}
= \{ x: \vert x \vert < q^{-(k-1)} \}
\]
for each $k \in \Z$; in particular, $\rho^k \cD$ is both open and closed for each $k \in \Z$.

The set $\cD$ is a ring, called the \textit{ring of integers} of $\cK$. 
Each of the sets $\rho^k \cD$, $k \in \Z$, is a compact $\cD$-submodule of $\cK$, and every non-trivial compact $\cD$-submodule of $\cK$ is of this form. 
For $\ell <k$ the additive quotient group $\rho^\ell \cD / \rho^k \cD$ has order $q^{k-\ell}$. Consequently, $\cD$ is the union of $q$ disjoint translates (that is, cosets) of $\rho \cD$. 
Each of these cosets is, in turn, the union of $q$ disjoint translates of $\rho^2 \cD$, and so on.  

\begin{remark}
We can thus think of the collection of balls contained in $\cD$ as being arranged in an infinite rooted $q$-ary tree: the root is $\cD$ itself, the nodes at level $k$ are the balls of radius $q^{-k}$ (cosets of $\rho^k \cD$), and the $q$ ``children'' of such a ball are the $q$ cosets of $\rho^{k+1} \cD$ that it contains. We can uniquely associate each point in $\cD$ with the sequence of balls that contain it, and so we can think of the points in $\cD$ as the boundary of this tree.
\end{remark}

\subsection{Norms}
\label{S:norms}

A \textit{norm} on a vector space $\cE$ over the local field $\cK$ is a real-valued mapping $\|\cdot\|$ on $\cE$ with the properties
\begin{align}
&\norm{x} = 0 \iff x=0, \\
&\norm{\lambda x} = |\lambda| \norm{x}, \quad  \lambda \in \cK, \\
&\norm{x+y} \leq \norm{x} \vee \norm{y}.
\end{align}
The norm induces a metric on $\cE$ by $(x,y) \mapsto \norm{x-y}$.
The resulting metric space is an ultrametric space. 
We always take the norm on $\cK^n$ to be the one given by
\[
\norm{(x_i)_{i \in [n]}} := \bigvee_{i \in [n]} |x_i|.
\]
The space $\cK^n$ is complete under this metric.  
In general, a \textit{$\cK$-Banach space} is a vector space $\cE$ over $\cK$ that is equipped with a norm such that the resulting metric space is complete.

\subsection{Orthogonality}
\label{S:orthogonality}

The following definition of orthogonality in a normed vector space over $\cK$ mimics a characterization of orthogonality in an inner product space over $\R$ that does not explicitly involve the inner product and instead is in terms of the induced metric.

\begin{definition}
Given a subset $G$ of a normed vector space $\cE$ over $\cK$, write $\langle G \rangle$ for the linear span of $G$.  
A subset $F$ of $\cE$ is \textit{$\cK$-orthogonal} if
$\|x - y\|  \ge \|x\|$ for all $x \in F$ and  $y \in \langle F \setminus \{x\} \rangle$; that is, the best approximation of $x$ in the vector space $\langle F \setminus \{x\} \rangle$ is $0$.
Equivalently, a subset $F$ is $\cK$-orthogonal if for any finite subset $\brac{x_1,\dots,x_n} \subseteq F$ and collection of scalars $\alpha_1, \ldots, \alpha_n \in \cK$,
\[
\norm{\sum_{i \in [n]} \alpha_i x_i} = \bigvee_{i \in [n]} |\alpha_i| \norm{x_i}.
\]
A subset $F$ of $\cE$ is \textit{$\cK$-orthonormal} if it is $\cK$-orthogonal and $\|x\| = 1$ for all $x \in F$.
\end{definition}

\begin{remark}
\label{R:orthonomality_equivalent}
If $F$ is $\cK$-orthonormal, then, 
for any finite subset $\brac{x_1,\dots,x_n} \subseteq F$ and collection of scalars $\alpha_1, \ldots, \alpha_n \in \cK$,
$
\norm{\sum_{i \in [n]} \alpha_i x_i} = \bigvee_{i \in [n]} |\alpha_i| \norm{x_i} = \bigvee_{i \in [n]} |\alpha_i|. 
$
Conversely, suppose that for any finite subset $\brac{x_1,\dots,x_n} \subseteq F$ and collection of scalars $\alpha_1, \ldots, \alpha_n \in \cK$ we have
$
\norm{\sum_{i \in [n]} \alpha_i x_i} =  \bigvee_{i \in [n]} |\alpha_i|. 
$
Taking $\alpha_j = 1$ and $\alpha_i = 0$ for $i \ne j$, $j \in [n]$,  we see that $\norm{x_j} = 1$ for $j \in [n]$ and hence 
$
\norm{\sum_{i \in [n]} \alpha_i x_i} =  \bigvee_{i \in [n]} |\alpha_i| \norm{x_i};
$
that is, $F$ is $\cK$-orthonormal.
\end{remark}

\begin{theorem}
\label{T:orthogonal_matrix_equivalents}
The following are equivalent for an $n \times n$ matrix $U$. 
\begin{itemize}
\item[(i)] The matrix $U$ is an isometry of $\cK^n$.
\item[(ii)] The matrix $U$ is invertible and the entries of both $U$ and $U^{-1}$ lie in $\cD$. 
\item[(iii)] The columns of $U$ are $\cK$-orthonormal.
\item[(iv)] The rows of $U$ are $\cK$-orthonormal.
\item[(v)] The entries of $U$ lie in $\cD$ and $|\det U| = 1$.
\end{itemize}
\end{theorem}

\begin{proof}
(i) $\iff$ (ii):  See Remark 3.2 in \cite{MR1934156}.

(i) $\iff$ (iii):  Note that
\begin{align*}
& \norm{Ux} = \norm{x},  
\quad \forall x \in \cK^n, \\
&\iff \bigvee_{i \in [n]} |(Ux)_i| = \bigvee_{i \in [n]} |x_i|, 
\quad \forall x \in \cK^n, \\
&\iff \bigvee_{i \in [n]} \abs{\sum_{j \in [n]} U_{i,j}x_j} = \bigvee_{i \in [n]} |x_i|,
 \quad \forall x \in \cK^n, \\
&\iff \norm{\sum_{j \in [n]} x_j (U_{i,j})_{i \in [n]}} = \bigvee_{i \in [n]} |x_i|,
 \quad \forall x \in \cK^n, \\
&\iff \text{the columns of $U$ are orthonormal},
\end{align*}
where we applied Remark~\ref{R:orthonomality_equivalent} for the last equivalence.

(i) $\iff$ (iv): We have already shown that (i), (ii), and (iii) are equivalent.
It remains to observe that $U$ is invertible with 
$U$ and $U^{-1}$ both having entries in $\cD$ if and only if the transpose
$U^\top$ is invertible with
$U^\top$ and $(U^\top)^{-1}$ both having entries in $\cD$.

(ii) $\iff$ (v):  If (ii) holds, then it follows from the properties of the valuation that
$|\det U| \le 1$ and $|\det U|^{-1} = |(\det U)^{-1}| = |\det U^{-1}| \le 1$, so that $|\det U| = 1$ and
(v) holds.  
Conversely, if (v) holds, then (ii) follows from Cramer's rule.
\end{proof}

\begin{notation}
In light of Theorem~\ref{T:orthogonal_matrix_equivalents}, we write
$\mathrm{GL}_n(\cD)$ for the group of matrices that satisfy the equivalent conditions of the
theorem and say that these matrices are $\cK$-orthogonal.
\end{notation}

\begin{notation}
For $n \in \N$, denote the unit sphere $\{x \in \cK^n : \norm{x} = 1\} = \cD^n \setminus (\rho \cD)^n$ by $\cS^{(n)}$.
\end{notation}

\begin{lemma}
\label{L:transitive}
The group of matrices $\mathrm{GL}_n(\cD)$ acts transitively on the set $\cS^{(n)}$;
that is, given $x,y \in \cK^n$ with $\norm{x} = \norm{y} = 1$, there is exists $U \in \mathrm{GL}_n(\cD)$
such that $U x = y$.
\end{lemma}

\begin{proof}
Let $\mathbf{e}_i$, $i \in [n]$, be the coordinate vectors in $\cK^n$; that is $\mathbf{e}_i$ is the vector with $1$ in the
$i^\mathrm{th}$ coordinate and $0$ elsewhere.
Because $\mathrm{GL}_n(\cD)$ is a group, it suffices to show that for any $x \in \cK^n$ with $\norm{x} = 1$ 
there exists a matrix $U \in \mathrm{GL}_n(\cD)$ such that $U \mathbf{e}_1 = x$.  

Fix such an $x$.  Because $\norm{x} = 1$, there is at least one coordinate, say $i$, such that $|x_i| = 1$.  Let
$U$ be an $n \times n$ matrix that has first column equal to $x$ and remain columns given by the $n-1$ vectors $\mathbf{e}_j$, $j \ne i$,
listed in some order.  Note that $U \mathbf{e}_1 = x$.  Clearly, $\det U = \pm x_i$ so that $|\det U| = |x_i| = 1$ and
$U \in \mathrm{GL}_n(\cD)$.
\end{proof}

\subsection{Haar measure}
\label{S:Haar_measure}

There is a unique measure $\lambda$ on $\cK$ which has the properties
\[
\lambda(x+A) = \lambda(A)
\]
and
\[
\lambda(xA) = |x| \lambda(A)
\]
for each $x \in \cK$,
 and 
\[
\lambda(\cD) = 1;
\]
the measure $\lambda$ is just the suitably normalized Haar measure on the additive group of $\cK$.

For $n \in \N$, the measure $\lambda^{\otimes n}$ on $\cK^n$ has the properties
\[
\lambda^{\otimes n}(x+A) = \lambda^{\otimes n}(A)
\]
for each $x \in \cK^n$ and 
\[
\lambda^{\otimes n}(M A) = |\det M| \, \lambda^{\otimes n}(A)
\]
for each $n \times n$ matrix $M$.  In particular, $\lambda^{\otimes n}$ is 
$\mathrm{GL}_n(\cD)$-invariant.

\begin{notation}
Write $\gamma$ for the restriction of the measure $\lambda$ to $\cD$ and, for $n\in \N$, set $\gamma_n := \gamma^{\otimes n}$.
Denote by $\sigma_n$ the probability measure obtained by conditioning the probability measure $\gamma_n$ on the set
$\cS^{(n)}$; that is, $\sigma_n$ is $\gamma_n(\cdot \cap \cS^{(n)})$ normalized to be a probability measure.
\end{notation}

\begin{proposition}
\label{P:invariant}
The probability measure $\sigma_n$ is the unique probability measure on $\cS^{(n)}$ that is invariant under the action of
$\mathrm{GL}_n(\cD)$.
\end{proposition}

\begin{proof}
It is clear that $\gamma_n$ is invariant under the action of $\mathrm{GL}_n(\cD)$.  
Because
$\cS^{(n)}$ is a subset of $\cD^n$ that has positive $\gamma_n$ measure (namely, $1 - q^{-n}$) that is invariant under the action of $\mathrm{GL}_n(\cD)$, it follows that $\sigma_n$ is invariant under the action of $\mathrm{GL}_n(\cD)$.
It therefore remains to establish the uniqueness claim.  This, however, is immediate from Theorem~\ref{T:invariant} below
and Lemma~\ref{L:transitive}.  
\end{proof} 

\begin{remark}
(i) The uniqueness claim in Proposition~\ref{P:invariant} may be established by an alternative route.
Because $\mathrm{GL}_n(\cD)$ is a compact, second countable, Hausdorff group, it possesses a Haar measure which is unique if we normalize it to be a probability measure.  Denote this probability measure by $\mu$.  Suppose that $\nu$ is a $\mathrm{GL}_n(\cD)$-invariant probability measure on $\cS^{(n)}$.
For any positive Borel function $f$ on $\cS^{(n)}$ and $U \in \mathrm{GL}_n(\cD)$ we have
$\int f(x) \, \nu(dx) = \int f(U x) \, \nu(dx)$ and so 
$\int f(x) \, \nu(dx) = \int \int f(U x) \, \nu(dx) \, \mu(dU) = \int \int f(U x) \, \mu(dU) \, \nu(dx)$.
From Lemma~\ref{L:transitive} we know for each $x \in \cS^n$ that there is a matrix $V \in \mathrm{GL}_n(\cD)$ such that $x = V \mathbf{e}_1$, where $\mathbf{e}_1 = (1,0,\ldots,0)^\top$.  
Therefore, by the invariance of $\mu$,  $ \int f(U x) \, \mu(dU) = \int f(U V \mathbf{e}_1) \, \mu(dU) = \int f(U  \mathbf{e}_1) \, \mu(dU)$ and hence $\int f(x) \, \nu(dx)$ is necessarily $\int f(U  \mathbf{e}_1) \, \mu(dU)$.

\noindent
(ii) It is possible to describe the Haar measure on $\mathrm{GL}_n(\cD)$ quite concretely.  Let $\gamma_{n \times n}$ be the probability measure on $n \times n$ matrices given by $\gamma_{n \times n}((dm_{i,j})_{i,j \in [n]}) = \bigotimes_{i,j \in [n]} \gamma(dm_{i,j})$; that is, if $M$ is distributed according to $\gamma_{n \times n}$, then the entries of $M$ are independent and each is distributed according to $\gamma$.  The Haar measure on $\mathrm{GL}_n(\cD)$ is just $\gamma_{n \times n}$ conditioned on the set $\mathrm{GL}_n(\cD)$; that is, the Haar measure is the restriction of $\gamma_{n \times n}$ to $\mathrm{GL}_n(\cD)$ normalized to be a probability measure.  Because we don't need this result in what follows, we leave the (straightforward) proof to the reader.  We note in passing that the $\gamma_{n \times n}$ measure of $\mathrm{GL}_n(\cD)$ is $\prod_{i \in [n]} (1 - q^{-i})$
-- see \cite[Theorem~4.1]{MR1934156}.

\noindent
(iii) As we have seen, the probability measure $\gamma_n$ is $\mathrm{GL}_n(\cD)$-invariant.  On the other hand, if the push-forward of $\gamma_n$ by an $n \times n$ matrix $M$ is again $\gamma_n$, then property (v) of Theorem~\ref{T:orthogonal_matrix_equivalents} holds and $M \in \mathrm{GL}_n(\cD)$.  We may therefore add another equivalent property to the list in Theorem~\ref{T:orthogonal_matrix_equivalents}.  A similar remark holds with $\gamma_n$ replaced by $\sigma_n$.
\end{remark}

For the sake of completeness, we state the following classical result on the existence and uniqueness of measures invariant under the
action of a group (see, for example, \cite[Theorem~2.29]{MR1876169}).

\begin{theorem}
\label{T:invariant}
Let $G$ be a locally compact, second countable, Hausdorff group of measurable transformations on a locally compact, second countable,
Hausdorff space $S$.  Suppose that $G$ acts properly (that is, the set $\{g \in G : g s \in K\} \subseteq G$ is compact for all $s \in S$
and compact $K \subseteq S$) and transitively (that is, given $s,t \in S$ there exists $g \in G$ such that $g s = t$).  
Then, up to normalization, there is a unique non-zero $G$-invariant Radon measure on $S$.
\end{theorem}

The following corollary is an easy consequence of Proposition~\ref{P:invariant}, but we include the proof for the sake of completeness.

\begin{corollary}
\label{C:polar_decomposition}
Let $\xi$ be a $\cK^n$-valued random variable
such that $U \xi$ has the same distribution as $\xi$ for all $U \in \mathrm{GL}_n(\cD)$. Possibly on some extension
$(\bar \Omega, \bar{\cF}, \bar \bP)$
of
$(\Omega, \cF, \bP)$
there is a $\sigma_n$-distributed, 
$\cS^{(n)}$-valued random variable $\vartheta$ and a $\{\rho^m : m \in \Z\} \cup \{0\}$-valued random variable
$R$ independent of $\vartheta$ such that $\xi = R \vartheta$.  Consequently, if $\nu$ is a $\mathrm{GL}_n(\cD)$-invariant probability measure on $\cK^n$, then $\nu$ is the push-forward of $\pi \otimes \sigma_n$ by the map $(r,x) \mapsto r x$ for some probability measure $\pi$ on $\{\rho^m : m \in \Z\} \cup \{0\}$.
\end{corollary}

\begin{proof}
Define events $A_m$, $m \in \Z$, by $A_m := \{\omega \in \Omega: \norm{\xi(\omega)} = q^{-m}\}$ and set
$A_\infty := \{\omega \in \Omega: \norm{\xi(\omega)} = 0\}$.  Let $(\bar \Omega, \bar{\cF})$ be $\Omega \times \cS^{(n)}$
equipped with the product of the $\sigma$-field $\cF$ and the Borel $\sigma$-field on $\cS^{(n)}$. 
 Let $\bar \bP:= \bP \otimes \sigma_n$.  Extend the definition of $\xi$ to $\bar \Omega$ by, with the usual abuse of notation,
taking $\xi(\omega, x)$ to be $\xi(\omega)$.  

Set $R := \rho^m$ on the event $A_m \times \cS^{(n)}$, $m \in \Z$,  and $R := 0$ on the event $A_\infty \times \cS^{(n)}$.  Put
\[
\vartheta(\omega,x) :=
\begin{cases}
\rho^{-m} \xi(\omega,x),&  \omega \in A_m, \, m \in \Z, \\
x,& \omega \in A_\infty. \\
\end{cases}
\]
It is clear that $\xi = R \vartheta$.  For any $r \in \{\rho^m : m \in \Z\} \cup \{0\}$ the
conditional distribution of $\vartheta$ given the event $\{R = r\}$ is obviously
invariant under the action of $\mathrm{GL}_n(\cD)$ and so, by Proposition~\ref{P:invariant}, the random variable
$\vartheta$ is independent of $R$ with distribution $\sigma_n$.
\end{proof}

A classical theorem often attributed to Poincar\'e says that for fixed $k  \in \N$ the distribution of the first $k$ coordinates of a point uniformly distributed over the sphere of radius $\sqrt{n}$ in $\R^n$ converges to the distribution  of a vector of $k$ independent standard Gaussian variables as $n \to \infty$.
See Section~6 of \cite{MR898502} for a discussion of the history of this result leading to the conclusion that a more appropriate attribution is to the work of Borel in \cite{Borel_06} (see also Chapter V of \cite{Borel_14}).
It is shown in \cite{MR898502} that indeed the total variation distance between the distribution of the first $k$ coordinates of a point uniformly distributed over the sphere of radius $\sqrt{n}$ in $\R^n$ and the distribution  of a vector of $k$ independent standard Gaussian variables converges to zero as $n \to \infty$ provided that $k = o(n)$.

The analogue of such a result in the local field setting is the following.  Note that the relevant total variation distance converges to zero as $n \to \infty$ regardless of the relative size of $k$ with respect to $n$.

\begin{theorem}
\label{T:total_variation}
For $n \in \N$ and $k \in [n]$, let $\sigma_{n,k}$ be the push-forward of $\sigma_n$ by the map that sends $(x_i)_{i \in [n]} \in \cK^n$ to $(x_i)_{i \in [k]} \in \cK^k$.  The total variation distance between  the probability measures $\sigma_{n,k}$ and $\gamma_k$ is ${q^{-n}(1-q^{-k})}/(1-q^{-n})$.
\end{theorem}

\begin{proof}
Let $(\xi_{n,i})_{i \in [n]}$ have distribution $\sigma_n$ and let $(\eta_i)_{i \in [n]}$ have distribution $\gamma_n$ so that $(\xi_{n,i})_{i \in [k]}$ has distribution $\sigma_{n,k}$ and $(\eta_i)_{i \in [k]}$
has distribution $\gamma_k$.

By definition, the distribution of $(\xi_{n,i})_{i \in [n]}$ is the same as the distribution of $(\eta_i)_{i \in [n]}$ conditioned on the event $\{\norm{(\eta_i)_{i \in [n]}} = 1\}$. Write $\delta_{n,i}$ for the indicator of the event $\brac{\abs{\xi_{n,i}} = 1}$ and $\epsilon_i$ for the indicator of the event $\brac{\abs{\eta_i} = 1}$. Let $\nu_0$ (resp. $\nu_1$) be $\gamma$ conditioned on the event $\brac{x \in \cK : \abs{x} < 1}$ (resp. $\brac{x \in \cK : \abs{x} = 1}$); that is, $\nu_0$ (resp. $\nu_1$) is the conditional distribution of $\eta_i$ given the event $\brac{\abs{\eta_i} < 1}$ (resp. $\brac{\abs{\eta_i} = 1}$).

If $(e_i)_{i \in [n]} \in \brac{0,1}^n \neq 0$, then the conditional distribution of $(\xi_{n,i})_{i \in [k]}$ given the event $\brac{(\delta_{n,i})_{i \in [k]}= (e_i)_{i \in [k]}}$ is the same as that of  $(\eta_i)_{i \in [k]}$ given the event $\{(\epsilon_k)_{i \in [k]} = (e_i)_{i \in [k]}\}$; both conditional distributions are $\bigotimes_{i \in [k]} \nu_{e_i}$.  It follows that the total variation distance we seek is the same as the total variation distance between the distribution of $(\delta_{n,i})_{i \in [k]}$ and the distribution of $(\epsilon_i)_{i \in [k]}$.  Furthermore, the conditional distribution of $(\delta_{n,i})_{i \in [k]}$ given the event $\{\sum_{i \in [k]} \delta_{n,i} = j\}$ is the same as the conditional distribution of $(\epsilon_i)_{i \in [k]}$ given the event $\{\sum_{i \in [k]} \epsilon_i = j\}$.  Thus, it further suffices to compute the total variation distance between the distribution of  $\sum_{i \in [k]} \delta_{n,i}$ and the distribution of  $\sum_{i \in [k]} \epsilon_i$.

Put $X = \sum_{i \in [n]} \epsilon_i$ and $Y = \sum_{i \in [k]} \epsilon_i$. Noting that the distribution of $\sum_{i \in [k]} \delta_{n,i}$ is the same as the conditional distribution of $Y$ given the event $\{X \ne 0\}$, the
total variation distance we seek is
\[
\frac{1}{2} \sum_{y=0}^k |\bP\{Y=y\} - \bP\{Y=y \, | \, X \ne 0\} |.
\]
For $y \neq 0$ we have
\[
\bP\{Y=y \, | \, X \ne 0\} 
= \frac{\bP\{Y=y, \, X \ne 0\}}{\bP\{X \ne 0\}}
= \frac{\bP\{Y=y\}}{\bP\{X \ne 0\}},
\]
and so
\[
\begin{split}
\sum_{y=1}^k |\bP\{Y=y\} - \bP\{Y=y \, | \, X \ne 0\} |
& = \left(\frac{1}{1 - q^{-n}} - 1\right) \sum_{y=1}^k\bP\{Y=y\} \\
& = \left(\frac{1}{1 - q^{-n}} - 1\right) (1 - \bP\{Y=0\}) \\
& = \frac{q^{-n}}{1 - q^{-n}}(1-q^{-k}). \\
\end{split}
\]
Also,
\[
\bP\{Y = 0\} = q^{-k}
\]
and
\[
\bP\{Y = 0 \, | \, X \ne 0\} =  \frac{q^{-k}(1 - q^{-(n-k)})}{1 - q^{-n}},
\]
so that
\[
\begin{split}
|\bP\{Y=0\} - \bP\{Y=0 \, | \, X \ne 0\} |
& = q^{-k} - \frac{q^{-k}(1 - q^{-(n-k)})}{1 - q^{-n}} \\
& = \frac{q^{-k} (q^{-(n-k)} - q^{-n})}{1 - q^{-n}}. \\
\end{split}
\]
Therefore,
\[
\begin{split}
\frac{1}{2} \sum_{y=0}^k |\bP\{Y=y\} - \bP\{Y=y \, | \, X \ne 0\} |
& = \frac{1}{2} \frac{q^{-n}(1-q^{-k}) + q^{-k} (q^{-(n-k)} - q^{-n})}{1 - q^{-n}} \\
& = \frac{1}{2} \frac{2 q^{-n} - 2 q^{-(n+k)}}{1 - q^{-n}} \\
& = \frac{q^{-n} (1 - q^{-k})}{1 - q^{-n}}, \\
\end{split}
\]
as claimed.
\end{proof}

\subsection{Fourier theory}
\label{S:Fourier}

Recall that a \textit{character} on a locally compact Abelian group $G$ with the group operation written additively is a map $\kappa: G \to \mathbb{T} := \{z \in \C: |z| = 1\}$ such that $\kappa(g+h) = \kappa(g) \kappa(h)$ for all $g,h \in G$.

It is possible to fix a character $\chi$ for $\cK$ such that $\chi$ restricted to the subgroup $\cD$ is trivial (that is, always takes the value $1$) while $\chi$ restricted to the subgroup $\rho^{-1} \cD$ is non-trivial.  
Fixing any such choice of the character $\chi$, an arbitrary character on $\cK$ is of the form $x \mapsto \chi(a x)$ for some $a \in \cK$.  More generally, an arbitrary character on $\cK^n$ is of the form $x \mapsto \chi(a \cdot x)$ for some $a \in \cK^n$, where $a \cdot x$ is the usual dot product of the vectors $a$ and $x$.

A probability measure $\mu$ on $\cK^n$ has a Fourier transform $\hat \mu(t) := \int \chi(t \cdot x) \, \mu(dx)$, $t \in \cK^n$.  For example,  $\hat \gamma_n = \ind_{\cD^n}$.

\subsection{Gaussian random variables}
\label{S:Gaussian}

For an introduction to the/an analogue of Gaussian probability measures on local fields and the proofs of any results stated without proof in this subsection, see  \cite{MR990478, MR1832433}.

The following definition mimics a standard definition/characterization of Gaussian random variables taking values in a Banach spaces over the real numbers.

\begin{definition}
Let $\cE$ be a separable $\cK$-Banach, and let $\xi$ be an $\cE$-valued random variable. 
The random variable $\xi$ is \textit{$\cK$-Gaussian} if whenever $\xi_1$ and $\xi_2$ are independent copies of $\xi$ and $(\alpha_{1,1},\alpha_{1,2}),(\alpha_{2,1},\alpha_{2,2}) \in \cK^2$ are $\cK$-orthonormal, then $(\xi_1,\xi_2)$ has the same distribution as $(\alpha_{1,1}\xi_1 + \alpha_{1,2}\xi_2, \alpha_{2,1}\xi_1 + \alpha_{2,2}\xi_2)$.
\end{definition}

\begin{theorem}
\label{T:Gaussian}
Let $\cE$ be a separable $\cK$-Banach space and let $\xi$ be an $\cE$-valued random variable.
\begin{itemize}
\item[(i)]
The random variable $\xi$ is $\cK$-Gaussian if and only if the distribution of $\xi$ is Haar measure on some compact $\cD$-module of $\cE$.  In particular, if $\cE = \cK$, then $\xi$ is $\cK$-Gaussian if and only if the distribution of $\xi$ is either the point mass at $0$ or the restriction of the Haar measure $\lambda$ to one of the sets $\rho^m \cD$, $m \in \Z$, normalized to be a probability measure.
\item[(ii)]
The random variable $\xi$ is $\cK$-Gaussian if and only if the $\cK$-valued random variable $T \xi$ is  $\cK$-Gaussian for all
$T \in \cE^*$, where  $\cE^*$ is the dual space of continuous, linear maps from $\cE$ to $\cK$.
\item[(iii)]
If $\xi$ is $\cK$-Gaussian, then the compact $\cD$-module in (i) is the set
\[
\brac{x \in \cE : |T x| \leq \|T \xi\|_\infty \quad \forall T \in \cE^*},
\]
where $\|\zeta\|_\infty := \ess \sup\{|\zeta(\omega)| : \omega \in \Omega\}$ for a $\cK$-valued random variable $\zeta$.
\item[(iv)]
The random variable $\xi$ is $\cK$-Gaussian if and only if $\|T \xi\|_\infty$ is finite for all $T \in \cE^*$ and, for all $t \in \cK$,
\[
\bP[\chi(t T \xi)] = 
\begin{cases}
1,& |t| \|T \xi \|_\infty \le 1, \\
0,& \text{otherwise}. \\
\end{cases}
\]
\end{itemize}
\end{theorem}

\begin{definition}
A $\cK$-valued random variable is \textit{standard} $\cK$-Gaussian if its distribution is $\gamma$.
\end{definition}

\section{Rotatable random sequences}
\label{S:Freedman}

The following analogue of Theorem~2 in \cite{MR898502} follows from a combination of Corollary~\ref{C:polar_decomposition} and Theorem~\ref{T:total_variation}.

\begin{proposition}
\label{P:finite_Freedman}
Suppose that $\nu$ is a $\mathrm{GL}_n(\cD)$-invariant probability measure on $\cK^n$ so that $\nu$ is the push-forward of $\pi \otimes \sigma_n$ by the map $(r,x) \mapsto r x$ for some probability measure $\pi$ on $\{\rho^m : m \in \Z\} \cup \{0\}$. The total variation distance between $\nu$ and the push-forward of $\pi \otimes \gamma_n$ by the map $(r,x) \mapsto r x$ is at most $q^{-n}$.
\end{proposition}

\begin{definition}
An infinite sequence $(\xi_i)_{i \in \N}$ of $\cK$-valued random variables is \textit{rotatable} if $U (\xi_i)_{i \in [n]}$ has the same distribution as $(\xi_i)_{i \in [n]}$ for every $n \in \N$ and $U \in \mathrm{GL}_n(\cD)$.
\end{definition}

\begin{theorem}
\label{T:Freedman}
A $\cK$-valued random infinite sequence $\xi = (\xi_i)_{i \in \N}$ is rotatable if and only if 
$\sup_{i \in \N} \abs{\xi_i}$ is almost surely finite and
(possibly on some extension of $(\Omega, \cF, \bP)$) 
$\xi_i = \tau \eta_i$, $i \in \N$, where
\begin{itemize}
\item
$\eta_1,\eta_2,\dots$ are independent, identically distributed, standard $\cK$-Gaussian random variables, 
\item
$\tau$ is the random variable taking values in $\brac{\rho^m : m \in \Z} \cup \{0\}$ given by
\[
\tau := \begin{cases}
\rho^m, &\sup_j \abs{\xi_j} = q^{-m}, \; m \in \Z, \\
0, &\sup_j \abs{\xi_j} = 0.
\end{cases}
\]
\item
the random variable $\tau$ is independent of the sequence $\eta = (\eta_i)_{i \in \N}$.
\end{itemize}
\end{theorem}

\begin{proof}
For each $n \in \N$, let
$$\tau_n := \begin{cases}
\rho^m, &\norm{(\xi_i)_{i \in [n]}} = q^{-m}, \; m \in \Z, \\
0, &\norm{(\xi_i)_{i \in [n]}} = 0,
\end{cases}$$
and let $(\tilde{\xi}_{n,i})_{i \in [n]}$ be distributed according to $\sigma_n$ and independent of $\tau_n$. Observe that, by Corollary~\ref{C:polar_decomposition}, $(\xi_i)_{i \in [n]}$ has the same distribution a $\tau_n (\tilde{\xi}_{n,i})_{i \in [n]}$. Now let $\tilde \eta$ be an infinite sequence of independent, identically distributed, standard $\cK$-Gaussian random variables independent of $(\tau_n)_{n \in \N}$. Writing $\cL(\zeta)$ for the distribution of a random variable $\zeta$ and $\|\cdot - \cdot\|_{\text{TV}}$ for the total variation distance, we have
\begin{align*}
\norm{\cL((\xi_i)_{i \in [k]})- \cL(\tau_n (\tilde{\eta}_i)_{i \in [k]})}_{\text{TV}} 
&= \| \cL(\tau_n (\tilde{\xi}_{n,i})_{i \in [k]}) - \cL(\tau_n(\tilde{\eta}_i)_{i \in [k]}) \|_{\text{TV}} \\
&\leq \| \cL((\tilde{\xi}_{n,i})_{i \in [k]}) - \cL((\tilde{\eta}_i)_{i \in [k]} )\|_{\text{TV}},
\end{align*}
which goes to $0$ as $n \to \infty$ by Theorem~\ref{T:total_variation}. So $\tau_n \tilde{\eta}$ certainly converges in distribution to $\xi$.

Now $\abs{\tau_n} = \norm{(\xi_i)_{i \in [n]}} = \bigvee_{i \in [n]} \abs{\xi_i}$ is increasing with $n \in \N$, and
\begin{align*}
0 &= \inf_{m \in \Z} \bP\{ \abs{\xi_1} > q^m \} \\
&= \inf_{m \in \Z} \lim_{n \in \N} \bP\{ \abs{\tau_n \tilde{\eta}_1} > q^m \} \\
&= \inf_{m \in \Z} \sup_{n \in \N} \bP\{ \abs{\tau_n \tilde{\eta}_1} > q^m \} \\
&= \inf_{m \in \Z} \sup_{n \in \N} \sum_{\ell=0}^{\infty} \bP\{ \abs{\tau_n} > q^{m+\ell} \} \bP \{ \abs{\tilde{\eta}_1} = q^{-\ell} \} \\
&= \sum_{\ell=0}^\infty \inf_{m \in \Z} \sup_{n \in \N} \bP\{ \abs{\tau_n} > q^{m+\ell} \} \bP \{ \abs{\tilde{\eta}_1} = q^{-\ell} \} \\
&= \sum_{\ell=0}^{\infty} \inf_{m \in \Z} \bP\left\{ \sup_{n \in \N} \abs{\tau_n} > q^{m+\ell} \right\} \bP \{ \abs{\tilde{\eta}_1} = q^{-\ell} \} \\
&= \bP \left\{ \sup_{n \in \N} \abs{\tau_n} = \infty \right\},
\end{align*}
so that $\tau_n \to \tau$ almost surely for some random variable $\tau$ taking values in $\brac{\rho^m : m \in \Z} \cup \brac{0}$.

Thus $\xi$ has the same distribution as $\tau \tilde{\eta}$. The almost sure result then follows from the transfer lemma (Corollary 6.11 from \cite{MR1876169}); we thus have random variables $\eta,\breve{\tau}$ such that $(\tau, \tilde{\eta})$ has the same distribution as $(\breve \tau, \eta)$ and $\breve{\tau} \eta = \xi$ almost surely. It remains to observe that
\begin{align*}
\abs{\tau} &= \sup_{n \in \N} \norm{(\xi_i)_{i \in [n]}} \\
&= \sup_{n \in \N} \norm{\breve{\tau}(\eta_i)_{i \in [n]}} \\
&= \sup_{n \in \N} \abs{\breve{\tau}} \norm{(\eta_i)_{i \in [n]}} \\
&= \abs{\breve{\tau}} \sup_{n \in \N} \norm{(\eta_i)_{i \in [n]}} \\
&=\abs{\breve{\tau}}
\end{align*}
almost surely,
so that $\breve{\tau} = \tau$ almost surely.
\end{proof}

There are several extensions and variants of Freedman's theorem in the literature -- see \cite{MR1211786} for a review.  One extension is to consider infinite random sequences where the individual entries take values in $\R^m$ for some $m \in \N$ or, more generally, in some separable Banach space over $\R$ as in \cite{MR520963}.  The analogue of the result in \cite{MR520963} holds in the local field setting, as we now explain.  We say that a random sequence $\xi$ with entries in a separable $\cK$-Banach space $\cE$ is \textit{rotatable} if for any $n \in \N$ and $U \in \mathrm{GL}_n(\cD)$ we have that $(\sum_{j \in [n]} U_{i,j} \xi_j)_{i \in [n]}$ has the same distribution as $(\xi_i)_{i \in [n]}$.  Write $\Gamma(\cE)$ for the set of Gaussian probability measures on $\cE$.  The set $\Gamma(\cE)$ is a closed subset of the Polish space of probability measures on the Polish space $\cE$, where we equip the space of probability measures with the topology of weak converence.  By Theorem~\ref{T:Gaussian}, $\Gamma(\cE)$ is in a bijective correspondence with the set of compact $\cD$-modules in $\cE$.  The analytic content of Theorem~\ref{T:Freedman} is that the distribution of a rotatable sequence of $\cK$-valued random variables is $\int_{\Gamma(\cK)} \mu^{\otimes \N} \, \pi(d\mu)$ for some probability measure on $\Gamma(\cK)$ and it is not too difficult to establish the following generalization that is a local field counterpart to the result in \cite{MR520963}.  We omit the proof.

\begin{theorem}
An infinite sequence $\xi$ of random variables in a separable $\cK$-Banach space $\cE$ is rotatable if and only if the distribution of $\xi$ is $\int_{\Gamma(\cE)} \mu^{\otimes \N} \, \pi(d\mu)$ for some probability measure $\pi$ on $\Gamma(\cE)$.
\end{theorem}

Other variants of Freedman's theorem involve considering infinite random sequences $\xi$ such that for all $n \in \N$ the random vector $U (\xi_i)_{i \in [n]}$ has the same distribution as $(\xi_i)_{i \in [n]}$ for all $U \in G_n$, where $G_n$ is some subgroup of the $n \times n$ orthogonal group that contains the subgroup of permutation matrices.  For example, \cite{MR626767} considers the case where $G_n$ is the subgroup that fixes the vector $(1,\ldots,1)^\top$ and shows that in this case $\xi = \alpha + \beta \eta$, where $\eta$ is an independent, identically distributed, standard Gaussian sequence and $(\alpha, \beta)$ is an independent $\R \times \R_+$-valued random vector.  It would be interesting to investigate the counterparts of such results in the local field setting, but we leave this to a future paper.

\section{A local field analogue of Schoenberg's theorem}
\label{S:Schoenberg}

In the case of real-valued rotatable random variables over $\R$, Freedman's theorem is equivalent to a result about nonnegative definite functions (see Theorem 1.31 and Appendix 4 in \cite{MR2161313}) which we recalled in the Introduction. 
Here, we provide an analogue of the latter result in the local field setting.

Recall that a function $\phi$ defined on an Abelian group $G$ is \textit{nonnegative definite} if $\sum_{i,j \in [n]} \phi(g_i - g_j) z_i \bar z_j \ge 0$ for all $g_1, \ldots, g_n \in G$, $z_1, \ldots, z_n \in \C$, and $n \in \N$.

\begin{theorem}
Let $\phi : \{0\} \cup \{q^m : m \in \Z\} \to \R$ be such that $\lim_{m \to \infty} \phi(q^{-m}) = \phi(0) = 1$. For $n \in \N$, define $\phi_n : \cK^n \to \R$ by
\[
\phi_n(x_1,\dots,x_n) := \phi(|x_1| \vee \cdots \vee |x_n|).
\]
Then $\phi_n$ is nonnegative definite for every $n \in \N$ if and only if $\phi$ is nonnegative and nonincreasing.
\end{theorem}

\begin{proof}
By Bochner's theorem for locally compact groups (see Section 36 of \cite{MR0054173}, also see \cite{MR0005741}), the function $\phi_n$ is nonnegative definite if there exists some $\cK^n$-valued random variable $\xi^{(n)}$ such that $\phi_n(t) = \bE[\chi(t \cdot \xi^{(n)})]$ for $t \in \cK^n$; that is, $\phi_n$ is the characteristic function of $\xi^{(n)}$. Using the definition of $\phi_n$, this is
$$\bE[\chi(t \cdot \xi^{(n)})] = \phi(\norm{t}).$$

We claim for $n \in \N$ and $m \in [n]$ that $(\xi_i^{(n)})_{i \in [m]}$ has the same distribution as $\xi^{(m)}$. Indeed, for $t \in \cK^m$,
\begin{align*}
\bE[\chi(t \cdot \xi^{(m)})] &= \phi(|t_1| \vee \cdots \vee |t_m|) \\
&= \phi(|t_1| \vee \cdots \vee |t_m| \vee 0 \vee \cdots \vee 0) \\
&= \bE[\chi((t_1,\dots,t_m,0, \ldots,0) \cdot \xi^{(n)})]
\end{align*}
and the uniqueness of characteristic functions completes the proof of the claim. By the Kolmogorov extension theorem, there exists an infinite $\cK$-valued random sequence $\xi$ such that $(\xi_i)_{i \in [n]}$ has the same distribution as $\xi^{(n)}$ for each $n \in \N$.

We next claim that the random sequence $\xi$ is rotatable. Given $n \in \N$ and
$U \in \mathrm{GL}_n(\cD)$,
\begin{align*}
\bE[\chi(t \cdot U (\xi_i)_{i \in [n]})] 
&= \bE[\chi(U^\top t \cdot (\xi_i)_{i \in [n]})] \\
&= \phi(\| U^\top t \|) \\
&= \phi(\| t \|) \\
&= \bE[\chi(t \cdot (\xi_i)_{i \in [n]})],
\end{align*}
which shows that $U (\xi_i)_{i \in [n]}$ has the same distribution as $(\xi_i)_{i \in [n]}$.

Using Theorem~\ref{T:Freedman}, we have that $\xi = \tau \eta$ almost surely, where $\eta$ is a sequence of independent, identically distributed, standard $\cK$-Gaussian random variables. So $\phi_n$ is nonnegative definite for each $n \in \N$ if and only if $\phi(\norm{t}) = \bE[\chi(t \cdot \tau(\eta_i)_{i \in [n]})]$ for each $n \in \N$. Thus
\begin{align*}
\phi(\|t\|) 
&= \mathbb{E}[\chi(t \cdot \tau(\eta_i)_{i \in [n}])] \\
&= \mathbb{E}[\chi((\tau t)\cdot (\eta_i)_{i \in [n]})] \\
&= \mathbb{E}[\ind_{\{\|\tau t\| \le 1\}}] \\
&= \begin{cases}
1,& \|t\| = 0,\\
\mathbb{P}\{|\tau| \le \frac{1}{\|t\|}\},& \|t\| \ne 0.\\
\end{cases}
\end{align*}
This is equivalent to $\phi$ having the desired properties.
\end{proof}

\bigskip
\noindent
{\bf Acknowledgments:} We thank Persi Diaconis and an anonymous referee for a number of helpful suggestions.

\providecommand{\bysame}{\leavevmode\hbox to3em{\hrulefill}\thinspace}
\providecommand{\MR}{\relax\ifhmode\unskip\space\fi MR }
\providecommand{\MRhref}[2]{%
  \href{http://www.ams.org/mathscinet-getitem?mr=#1}{#2}
}
\providecommand{\href}[2]{#2}


\begin{thebibliography}{DEL92}

\bibitem[Bor06]{Borel_06}
E.~Borel, \emph{Sur les principes de la th\'eorie cin\'etique des gaz}, Annales
  scientifiques de l'\'E.N.S. $3^e$ serie \textbf{23} (1906), 9--32.

\bibitem[Bor14]{Borel_14}
\bysame, \emph{Introduction g\'eom\'etrique \`a quelques th\'eories physiques},
  Gauthier-Villars, Paris, 1914.

\bibitem[Daw78]{MR520963}
A.~P. Dawid, \emph{Extendibility of spherical matrix distributions}, J.
  Multivariate Anal. \textbf{8} (1978), no.~4, 559--566. \MR{520963}

\bibitem[DEL92]{MR1211786}
Persi~W. Diaconis, Morris~L. Eaton, and Steffen~L. Lauritzen, \emph{Finite de
  {F}inetti theorems in linear models and multivariate analysis}, Scand. J.
  Statist. \textbf{19} (1992), no.~4, 289--315. \MR{1211786}

\bibitem[DF87]{MR898502}
Persi Diaconis and David Freedman, \emph{A dozen de {F}inetti-style results in
  search of a theory}, Ann. Inst. H. Poincar\'{e} Probab. Statist. \textbf{23}
  (1987), no.~2, suppl., 397--423. \MR{898502}

\bibitem[Eat81]{MR606622}
Morris~L. Eaton, \emph{On the projections of isotropic distributions}, Ann.
  Statist. \textbf{9} (1981), no.~2, 391--400. \MR{606622}

\bibitem[Eva89]{MR990478}
Steven~N. Evans, \emph{Local field {G}aussian measures}, Seminar on
  {S}tochastic {P}rocesses, 1988 ({G}ainesville, {FL}, 1988), Progr. Probab.,
  vol.~17, Birkh\"{a}user Boston, Boston, MA, 1989, pp.~121--160. \MR{990478}

\bibitem[Eva01]{MR1832433}
\bysame, \emph{Local fields, {G}aussian measures, and {B}rownian motions},
  Topics in probability and {L}ie groups: boundary theory, CRM Proc. Lecture
  Notes, vol.~28, Amer. Math. Soc., Providence, RI, 2001, pp.~11--50.
  \MR{1832433}

\bibitem[Eva02]{MR1934156}
\bysame, \emph{Elementary divisors and determinants of random matrices over a
  local field}, Stochastic Process. Appl. \textbf{102} (2002), no.~1, 89--102.
  \MR{1934156}

\bibitem[Fre62]{MR0156369}
David~A. Freedman, \emph{Invariants under mixing which generalize de
  {F}inetti's theorem}, Ann. Math. Statist \textbf{33} (1962), 916--923.
  \MR{0156369}

\bibitem[Kal02]{MR1876169}
Olav Kallenberg, \emph{Foundations of modern probability}, second ed.,
  Probability and its Applications (New York), Springer-Verlag, New York, 2002.
  \MR{1876169}

\bibitem[Kal05]{MR2161313}
\bysame, \emph{Probabilistic symmetries and invariance principles}, Probability
  and its Applications (New York), Springer, New York, 2005. \MR{2161313}

\bibitem[Kel70]{MR0287628}
Douglas Kelker, \emph{Distribution theory of spherical distributions and a
  location-scale parameter generalization}, Sankhy\={a} Ser. A \textbf{32}
  (1970), 419--438. \MR{0287628}

\bibitem[Kin72]{MR0343420}
J.~F.~C. Kingman, \emph{On random sequences with spherical symmetry},
  Biometrika \textbf{59} (1972), 492--494. \MR{0343420}

\bibitem[Let81]{MR606624}
G\'{e}rard Letac, \emph{Isotropy and sphericity: some characterisations of the
  normal distribution}, Ann. Statist. \textbf{9} (1981), no.~2, 408--417.
  \MR{606624}

\bibitem[Loo53]{MR0054173}
Lynn~H. Loomis, \emph{An introduction to abstract harmonic analysis}, D. Van
  Nostrand Company, Inc., Toronto-New York-London, 1953. \MR{0054173}

\bibitem[Max75]{Maxwell_75}
J.C. Maxwell, \emph{Theory of heat}, 4th ed., Longmans, London, 1875.

\bibitem[Max78]{Maxwell_78}
\bysame, \emph{On {B}oltzmann's theorem on the average distribution of energy
  in a system of material points}, Trans. Cambridge Phil. Soc. \textbf{12}
  (1878), 547.

\bibitem[Sch38]{MR1501980}
I.~J. Schoenberg, \emph{Metric spaces and positive definite functions}, Trans.
  Amer. Math. Soc. \textbf{44} (1938), no.~3, 522--536. \MR{1501980}

\bibitem[Sch06]{MR2444734}
W.~H. Schikhof, \emph{Ultrametric calculus}, Cambridge Studies in Advanced
  Mathematics, vol.~4, Cambridge University Press, Cambridge, 2006, An
  introduction to $p$-adic analysis, Reprint of the 1984 original [MR0791759].
  \MR{2444734}

\bibitem[Smi81]{MR626767}
A.~F.~M. Smith, \emph{On random sequences with centred spherical symmetry}, J.
  Roy. Statist. Soc. Ser. B \textbf{43} (1981), no.~2, 208--209. \MR{626767}

\bibitem[Tai75]{MR0487295}
M.~H. Taibleson, \emph{Fourier analysis on local fields}, Princeton University
  Press, Princeton, N.J.; University of Tokyo Press, Tokyo, 1975. \MR{0487295}

\bibitem[vR78]{MR512894}
A.~C.~M. van Rooij, \emph{Non-{A}rchimedean functional analysis}, Monographs
  and Textbooks in Pure and Applied Math., vol.~51, Marcel Dekker, Inc., New
  York, 1978. \MR{512894}

\bibitem[Wei40]{MR0005741}
Andr\'e Weil, \emph{L'int\'egration dans les groupes topologiques et ses
  applications}, Actual. Sci. Ind., no. 869, Hermann et Cie., Paris, 1940,
  [This book has been republished by the author at Princeton, N. J., 1941.].
  \MR{0005741}

\end{thebibliography}
\end{document}